\documentclass{amsart}

\usepackage[all]{xypic}
\usepackage{verbatim}

\newtheorem{thrm}{Theorem}[section]
\newtheorem{prop}[thrm]{Proposition}
\newtheorem{lemma}[thrm]{Lemma}

\theoremstyle{definition}
\newtheorem{defn}[thrm]{Definition}

\theoremstyle{remark}
\newtheorem{ex}[thrm]{Example}
\newtheorem{rmk}[thrm]{Remark}
\newtheorem{notation}[thrm]{Notation}

\newcommand{\sma}{\wedge}

\newcommand{\I}{\mathcal{I}}
\newcommand{\Top}{\mathcal{T}\!\!\textit{op}}
\newcommand{\col}{\colon}
\newcommand{\grps}{\mathcal{G}}
\DeclareMathOperator{\Hom}{Hom}
\newcommand{\RR}{\mathbb{R}}
\newcommand{\KK}{\mathbb{K}}
\newcommand{\ZZ}{\mathbb{Z}}
\newcommand{\CC}{\mathbb{C}}
\newcommand{\UU}{\mathbb{U}}
\newcommand{\PP}{\mathbb{P}}
\newcommand{\Cat}{\mathcal{C}\!\textit{at}}
\newcommand{\csphere}{\mathfrak{s}}
\newcommand{\calK}{\mathcal{K}}
\newcommand{\inv}{{-1}}

\DeclareMathOperator{\Spin}{Spin}
\DeclareMathOperator{\id}{id}
\newcommand{\MSpin}{\mathit{MSpin}}
\newcommand{\MMSpin}{\mathbb{M}\mathit{Spin}}
\newcommand{\calC}{\mathcal{C}}
\newcommand{\calV}{\mathcal{V}}

\newcommand{\IsubGspa}{\mathcal{I}_G\textrm{--}\mathit{spaces}}
\DeclareMathOperator{\Id}{Id}
\newcommand{\IGspa}{\mathcal{I}\textrm{--}G\mathit{spaces}}
\newcommand{\Igrpspa}{\mathcal{I}_{\grps}\textrm{--}\mathit{spaces}}
\newcommand{\Ispa}{\mathcal{I}\textrm{--}\mathit{spaces}}
\newcommand{\xarr}[1]{\xrightarrow{#1}}

\let\oldmarginpar\marginpar
\renewcommand\marginpar[1]{\-\oldmarginpar[\raggedleft\footnotesize #1]%
{\raggedright\footnotesize #1}}

\makeatletter
\let\c@equation\c@thrm
\makeatother
\numberwithin{equation}{section}
\begin{document}
\title{Global Orthogonal Spectra}
\author{Anna Marie Bohmann}
%\date{\today}
\maketitle

\begin{abstract}
For any finite group $G$, there are several well-established definitions of a $G$-equivariant spectrum.  In this paper, we develop the definition of a global orthogonal spectrum.  Loosely speaking, this is a coherent choice of orthogonal $G$-spectrum for each finite group $G$.  We use the framework of enriched indexed categories to make this precise.  We also consider equivariant $K$-theory and $\Spin^c$-cobordism from this perspective, and we show that the Atiyah--Bott--Shapiro orientation extends to the global context.
\end{abstract}

\section{Introduction}\label{sect:intro}

 Let $\grps$ be the category of finite groups and homomorphisms.  For a fixed group $G\in \grps$, we may study $G$-spaces using the techniques and methods of $G$-equivariant algebraic topology.  Many classical pieces of algebraic topology generalize to the equivariant world, although calculations are significantly  more difficult in this context.  We propose to generalize to a ``global'' equivariant world, in which we have compatible $G$-spectra across the entire category of groups.  Several classical spectra in topology do generalize naturally to this context, including complex cobordism and complex $K$-theory.  

The idea of global orthogonal spectra was inspired by the paper \cite{GM1997}, in which Greenlees and May introduce the concept of global $\I_*$-functors with smash product.  The constructions of the present work use a substantially different categorical framework from that of \cite{GM1997} so as to make the work fit into modern understandings of point-set level spectra.  

Equivariant orthogonal spectra, as defined and discussed by Mandell and May \cite{MM2002}, are  one of several good point-set level categories of equivariant spectra.  In this paper, we refine the notion of an equivariant orthogonal spectrum given in \cite{MM2002} to a global version.  We focus on the case of finite groups and genuine $G$-spectra.  In the remainder of this section, we recall the definition of an orthogonal $G$-spectrum.  In Section \ref{sec:indexedcattheory}, we define the categorical structure used in our model of global spectra; in Section \ref{sec:defnglobalspec}, we define this model.  Section \ref{sec:ktheoryspincobord} discusses the examples of complex $K$-theory and $\Spin^c$ cobordism after a paper of Joachim \cite{Joachim2004}.  Finally, in Appendices \ref{singlegroup} and \ref{sec:globalize-this}, we compare some of the categories in our construction to related categories arising in modeling spectra.

\subsection{Basic Definitions}\label{basicdefn}

Fix a compact Lie group $G$. Orthogonal $G$-spectra are usually defined in terms of diagrams on a category whose objects are representations of $G$ and whose morphisms are linear isometric isomorphisms; see \cite[Chapter 2]{MM2002}.  Although the category of all $G$-representations is large, it has a small skeleton.  We will define orthogonal $G$-spectra in terms of a particularly well-behaved skeleton of the category of representations; the author thanks Peter May for suggesting this nicely streamlined approach.

Let $\I_G$ be the category whose objects are pairs $(\RR^n, \rho\col G\to O(n))$, where $\rho$ is a homomorphism from $G$ to $O(n)$ and thus endows $\RR^n$ with the structure of a $G$-representation.  Morphisms $(\RR^m,\mu)\to (\RR^n,\rho)$ are linear isometric isomorphisms $\RR^m\to \RR^n$; hence, the group of automorphisms of $(\RR^n,\rho)$ is $O(n)$.  This means that $\I_G$ is equivalent to the usual category for indexing genuine orthogonal $G$--spectra, and we will sometimes refer to objects of $\I_G$ as ``representations.''  Since we do not require morphisms to commute with the homomorphisms $G\to O(n)$, we have a $G$-action by conjugation on the hom-spaces of $\I_G$; concretely, given $f\col (\RR^m,\mu)\to (\RR^n,\rho)$,  an element $g\in G$ acts on $f$ according to the formula
\[g.f= \RR^m \xarr{\mu(g^\inv)} \RR^m\xarr{f}\RR^n\xarr{\rho(g)}\RR^n.\]

Similarly, define $\Top_G$ to be the category whose objects are based $G$-spaces and whose morphisms are continuous based maps.  The category $\Top_G$ is enriched in based $G$-spaces via conjugation. We implicitly add disjoint basepoints to the morphism spaces in $\I_G$ so that it is also enriched in based $G$-spaces.  We denote the category of based $G$-spaces and equivariant maps by $G\Top$, so both $\Top_G$ and $\I_G$ are $G\Top$-enriched categories.

\begin{rmk} We recall the relevant terminology of enriched category theory.  By a \emph{$G$--continuous functor} between $G\Top$--enriched categories, we mean a functor preserving the $G$--space structure on hom spaces: that is, for a $G$-continuous functor $X$ and a morphism $f$, we require $gX(f)g^\inv=X(gfg^\inv)$ for all $g\in G$.  A natural $G$-transformation between functors $X,Y\col \mathcal{C}\to \mathcal{D}$ is a collection of $G$--equivariant maps $\nu\col X(C)\to Y(C)$ that commute with the morphisms in $\mathcal{C}$:
\[
\xymatrix{X(C)\ar[r]^{\nu}\ar[d]_{X(f)} & Y(C)\ar[d]_{Y(f)}\\
X(D) \ar[r]^{\nu} & Y(D)}
\]
We emphasize that $X(f)$ and $Y(f)$ need not be equivariant maps, but the components $\nu$ must be equivariant.
\end{rmk}

\begin{notation}
In order to avoid cluttering notation as much as possible, we generally omit explicit mention of the component of a natural transformation such as $\nu$ in our notation.  For example, we write $\nu\col X(C)\to Y(C)$ rather than $\nu_C$.
\end{notation}

\begin{defn}
An \emph{$\I_G$-space} is a $G$-continuous functor $X\col \I_G\to \Top_G$.  Morphisms between $\I_G$-spaces are natural $G$-transformations. 
\end{defn}
Requiring the morphisms of $\I_G$-spaces to be natural $G$-transformations with equivariant components means that the spaces of morphisms between $\I_G$-spaces do not have a $G$-action.

\begin{ex} The one-point compactification functor $V\mapsto S^V$ gives the sphere $\I_G$-space, which we denote by $S$.
\end{ex}

\begin{ex}\label{ex:Gcxcobordism}
Let $V\in\I_G$ be an $n$-dimensional $G$-representation, and let $MU(V)=T(BU_n(V\oplus V))$ be the Thom space of the canonical bundle of $n$-planes in $V\oplus V$.  The functor $MU_G\col \I_G\to \Top_G$ given by $V\mapsto MU(V)$ is an $\I_G$-space.
\end{ex}

\begin{defn} An \emph{orthogonal $G$-spectrum} is an $\I_G$-space $X$ with a natural transformation of functors $\I_G\times \I_G\to \Top_G$ \[ X(-)\sma S^{(-)}\to X(-\oplus-)\]
satisfying appropriate associativity and unitality diagrams.  In other words, an orthogonal $G$-spectrum is an $\I_G$-space with an action of the sphere $\I_G$-space.
\end{defn}

\begin{ex} The sphere $\I_G$-space $S$ is actually an orthogonal $G$-spectrum, since we have $G$-homeomorphisms $S^V\sma S^W\cong S^{V\oplus W}$ for any $G$-representations $V$ and $W$.  In fact, $S$ is a strong symmetric monoidal functor from $\I_G$ to $\Top_G$.
\end{ex}

\begin{ex} The functor $MU_G$ of Example \ref{ex:Gcxcobordism} is also an orthogonal $G$-spectrum.  The map $MU(V)\sma S^W\to MU(V\oplus W)$ arises from the Thom construction applied to the map $BU_n(V\oplus V)\to BU_{n+m}(V\oplus W\oplus V\oplus W)$ given by direct sum of an $n$-plane with the basepoint copy of $W$.    As in \cite{GM1997}, this spectrum is a model for $G$-equivariant complex cobordism. In fact $MU_G$ is also a lax symmetric monoidal functor, and thus an orthogonal ring spectrum.
\end{ex}

In the sequel, we generalize these definitions to the global case.  Heuristically, we wish to have some sort of ``fibration'' of equivariant spectra over the category $\grps$ of finite groups, so that the fiber over any $G\in \grps$ is a $G$-spectrum.  To make this idea precise, we need to introduce a categorical concept.

\section{Enriched indexed categories}\label{sec:indexedcattheory}

The main difficulty with generalizing the above definition of equivariant orthogonal spectra is understanding how to systematically treat the enrichments.  Varying the group of equivariance requires varying the category in which our representations or topological spaces are enriched: there is a natural $G$-structure on maps between $G$-spaces, but a natural $H$-structure on maps between $H$-spaces.  Fortunately, Mike Shulman \cite{Shulman2011} has systematically studied this type of varying enrichment under the name of ``enriched indexed categories.''

An enriched indexed category can be described in two equivalent ways.  For both, we require a cartesian monoidal indexing category $S$ and a monoidal fibration $\mathcal{V}\col V\to S$.  In other words, $V$ is a monoidal category, $\mathcal{V}$ is a strict monoidal functor and the monoidal product in $V$ preserves cartesian arrows, ie the fibrational structure is preserved by the product.
The fibers in $V$ will be the enriching categories.  In our case, indexing category $S$ is the category $\grps$ and $V$ is  the category $\Top_\grps$ of based spaces with an action of some group $G\in \grps$; the fibration $\Top_\grps\to \grps$ sends a space with a finite group action to the group of equivariance.  We denote the fiber of $\mathcal{V}$ over an object $X\in S$ by $\mathcal{V}_X$.  If $f\col X\to Y$ is a map in $S$, the structure of the fibration $\mathcal{V}\col V\to S$ provides a functor $f^*\col \mathcal{V}_Y\to \mathcal{V}_X$.  For any category $\mathcal{C}$ enriched in $\mathcal{V}_Y$, the notation $(f^*)_\bullet\mathcal{C}$ indicates application of the functor $f^*$ to the hom-objects of $\mathcal{C}$ so as to change the enriching category to $\mathcal{V}_X$.

\begin{defn}[{\cite{Shulman2011}}]
  An \emph{indexed $\mathcal{V}$-category} $\mathcal{C}$ consists of the following structure.
\begin{itemize}
\item[(a)] For each $X\in S$, a category $\mathcal{C}_X$ enriched in $\mathcal{V}_X$.
\item[(b)] For each morphism $f\col X\to Y$ in $ S$, a full-and-faithful $\mathcal{V}_X$-enriched functor $f^*\col (f^*)_{\bullet}(\mathcal{C}_Y)\to \mathcal{C}_X$
\item[(c)] For each composable pair of morphisms $X\xrightarrow{f} Y\xrightarrow{g} Z$ in $S$, a $\mathcal{V}_X$-natural isomorphism $(gf)^*\cong f^*\circ (f^*)_\bullet g^*$
\item[(d)] For each composable triple, the isomorphism in (c) should satisfy the appropriate compatibility condition.
\item[(e)] A unit isomorphism $\Id_{\mathcal{C}_X}\cong \id_X^*$ together with appropriate compatibility conditions.
\end{itemize}
\end{defn}

That is, for each indexing object $X\in S$, we have a category enriched in $\mathcal{V}_X$.  These categories vary coherently in $S$, and the functors relating them preserve as much of the enrichment as possible.  The categories $\I_G$ and $\Top_G$ defined in Section~\ref{sect:intro} fit together to form indexed $\Top_\grps$-categories.  
\begin{ex} $\Top_\grps$ is an indexed $\Top_\grps$-category in which the $\Top_G$-category corresponding to $G\in \grps$ is $\Top_G$ itself.  For a homomorphism $\alpha\col H\to G$ in $\grps$, the functor $(\alpha^*)_\bullet(\Top_G)\to \Top_H$ is given by restriction along $\alpha$.  These functors satisfy the required compatibility data.
\end{ex}
\begin{ex} Let $\I_\grps$ denote the indexed $\Top_\grps$-category in which the $\Top_G$-category over the index $G\in \grps$ is  $(\I_\grps)_G=\I_G$.  For each homomorphism $\alpha\col H\to G$ in $\grps$, the functor $\alpha^*\col (\alpha^*)_\bullet(\I_G)\to \I_H$ is restriction along $\alpha$, sending $(\RR^n,\rho)\in \I_G$ to $(\RR^n, \rho\circ\alpha)$.  A direct check shows this satisfies the required compatibility.
\end{ex}

Shulman proves that the following definition of a $\mathcal{V}$-fibration is equivalent to the definition of an indexed $V$-category. We follow Shulman's notation.
\begin{defn}[{\cite{Shulman2011}}] 
 Given a monoidal fibration $\mathcal{V}\col V\to S$, a \emph{large $\mathcal{V}$-category} $\mathcal{C}$ consists of the following structure. 
\begin{itemize}
\item[(a)] A collection of objects $x, y, z,\dotsc$
\item[(b)] For each $x$, an `extent' $\epsilon x\in S$.
\item[(c)] For each pair $x, y$, a hom-object $\underline{\mathcal{C}}(x,y)$ in $\calV_{\epsilon x\times\epsilon y}$
\item[(d)] For each $x$ an identities map $\I_{\epsilon x}\to \Delta^*\underline{\calC}(x,x)$ where $\I_{\epsilon x}$ is the unit of $\calV_{\epsilon x}$
\item[(e)] For each $x, y$ and $z$, a composition map
\[\underline{\calC}(y,z)\otimes_{\epsilon y}\underline{\calC}(x,y)\to \pi_2^*\underline{\calC}(x,z)\]
where $\pi_2$ is projection from $\epsilon x\times \epsilon y\times \epsilon z$ to $\epsilon x\times\epsilon z$.
\item[(f)] Composition must be associative and unital in the appropriate sense.
\end{itemize}
\end{defn}
We think of this structure as  a category that is in some sense `fibered over' $S$, equipped with a suitable fiber-wise enrichment.

\begin{ex} The category $\Top_\grps$ is also a large $\Top_\grps$-category.  This the indexed version of viewing $\Top_G$ as a $\Top_G$-enriched category.  An object of $\Top_\grps$ is a space $X$ with an action by a group $G\in \grps$; the group $G$ is its extent.  For a $G$-space $X$ and an $H$-space $Y$, the morphism object $\underline{\Top}_\grps(X,Y)=F(X,Y)$ is the $G\times H$-space of based continuous functions $X\to Y$ where the $G\times H$-action is given by conjugation in the usual way.  The identities map is the $G$-equivariant map $S^0\to F(X,X)$ sending the nonbasepoint of the $G$-fixed space $S^0$ to the identity map of $X$.  
\end{ex}

\begin{ex} The category $\I_\grps$ is also a large $\Top_\grps$-category.  Here, the objects  are pairs $V=(\RR^n,\rho)\in \I_G$ where $\rho\col G\to O(n)$ for some group $G\in \grps$; such a representation has extent $G$.  For $V\in \I_G$ and $W\in \I_H$, the morphism object $\underline{\I}_\grps(V,W)$ is the $G\times H$-space of linear isometric isomorphisms $V\to W$ which again has a $G\times H$-action by conjugation. The identities maps takes the nonbasepoint of $S^0$ to the identity element of $O(n)=\underline{\I}_\grps(V,V)$.
\end{ex}

\begin{defn}[{\cite{Shulman2011}}] 
A \emph{$\mathcal{V}$-fibration} $\mathcal{C}$ is a large $\mathcal{V}$-category such that for each object  $x$ and each morphism $f\col Y\to \epsilon x$ in $S$, there exists a restriction $f^*x$, ie an object $f^*x$ such that $\epsilon (f^*x)=Y$ and there is a natural isomorphism
\[ \underline{\mathcal{C}}(-,f^*x)\cong (f\times 1)^*\underline{\mathcal{C}}(-,x).\]
Here `natural' should be understood in a suitable $\calV$-categorical sense so that we allow for the change of enrichments when the objects in the blanks vary in extent.
\end{defn}

\begin{ex} Both $\I_\grps$ and $\Top_\grps$ are in fact $\Top_\grps$-fibrations.  In $\Top_\grps$, the restriction of a $G$-space $X$ along a homomorphism $\alpha\col H\to G$ is the $H$-space $\alpha^*X$; in $\I_\grps$, the restriction of a $G$-representation $(\RR^n, \rho)$ along $\alpha$  is given by precomposition with $\alpha$.  The required natural isomorphism, for example in the case of $\Top_\grps$, says that for any space $Z$ with an action of a finite group $K$, we have an isomorphism of $H\times K$ spaces
\[ F(Z,\alpha^*X)\cong (\alpha\times 1)^*F(Z,X)\]
that is natural in $Z$.  
\end{ex}

Shulman defines functors and natural transformations of $\mathcal{V}$-fibrations and proves a precise version of the following theorem.
\begin{thrm}[{\cite{Shulman2011}}]
\label{equivalenceofindexedcats}
The 2-categories of $\mathcal{V}$-fibrations and indexed $\mathcal{V}$-categories are equivalent.
\end{thrm}
This equivalence is analogous to the equivalence between functors $J\to \Cat$ and fibrations of categories with base $J$ given by the Grothendieck construction.  Given an indexed $\mathcal{V}$-category $\mathcal{C}$, the objects with extent $X$ are the objects of the category $\mathcal{C}_X$ indexed by $X$; if $x\in \mathcal{C}_X$ and $y\in \mathcal{C}_Y$, the hom-object $\underline{\mathcal{C}}(x,y)$ is given by $\mathcal{C}_{X\times Y}( \pi_1^*x,\pi_2^*y)$.  In the other direction, if $\mathcal{C}$ is a $\mathcal{V}$-fibration, the corresponding indexed $\mathcal{V}$-category consists of the $\mathcal{V}_X$-categories $\mathcal{C}_X$ whose objects are the objects with extent $X$ and whose morphisms are given by $\mathcal{C}_X(x,x')=\Delta^*\underline{\mathcal{C}}(x,x')$ for $x,x'\in \mathcal{C}$ with extent $X$.

Theorem \ref{equivalenceofindexedcats} allows us to think of both  $\Top_\grps$ and $\I_\grps$ as fibrations over $\grps$, with enrichment varying by fiber, and also as collections of indexed categories over each object $G\in \grps$, together with appropriate compatibility data for homomorphisms $\alpha\in \grps$.  In general, our conceptual picture will be based on regarding $\Top_\grps$ and $\I_\grps$ as categories fibered over $\grps$, but we will give concrete definitions using their descriptions as indexed $\Top_\grps$-categories.

Shulman also provides a framework for considering monoidal $\mathcal{V}$-fibrations when $\calV$ is symmetric monoidal. The 2-category of $\mathcal{V}$-fibrations has a monoidal product $\otimes_f$ (`$f$' for `fiberwise') which is given by a fiberwise product of the categories over each extent.  Explicitly, if $C$ and $D$ are $\mathcal{V}$-fibrations, then the category of $C\otimes_fD$ living over an extent $X$ is the product category $C_X\otimes D_X$.  Hence we have the following definition.
\begin{defn}\label{monoidalfibrdefn}
A \emph{monoidal $\mathcal{V}$-fibration} is a monoid in the 2-category of $\mathcal{V}$-fibrations under $\otimes_f$.
\end{defn}   
Unraveling this definition shows that a monoidal $\mathcal{V}$-fibration is just a $\mathcal{V}$-fibration  $\mathcal{C}$ such that each $\mathcal{V}_X$-category $\mathcal{C}_X$ is monoidal and the transition functors $f^*$ together with their coherences are strong monoidal.

Note that one can also define a non-fiberwise monoidal product on $\mathcal{V}$-categories where objects are just pairs of objects, regardless of whether these have the same extent. There is a symmetric monoidal biequivalence between $\mathcal{V}$-fibration with the fiberwise product and $\mathcal{V}$-categories with this external tensor product \cite{Shulman2011}.
Hence it suffices to consider the fiberwise tensor product, which is more suitable for our indexed set up.

\section{Defining global spectra}\label{sec:defnglobalspec}
Our definition of global orthogonal spectra runs essentially parallel to the definition of orthogonal $G$-spectra recalled in Section \ref{basicdefn}, except everything will be fibered over $\grps$.  To emphasize that $\grps$-enrichment has replaced $G$-enrichment, we call our global spectra `orthogonal $\grps$--spectra.'
\begin{rmk} For categorical reasons, we require that the indexing category be cartesian closed.  Throughout, we will use the category $\grps$ of all finite groups, but we could easily restrict to finite $p$-groups for a prime $p$.  While the category of all compact Lie groups is cartesian closed, we restrict to finite groups at present because of difficulties with finding coherent ways of restricting representations.
\end{rmk}

\begin{defn}\label{defn:Igrpspa}  An \emph{$\I_\grps$--space} is a functor of enriched indexed categories $A\col\I_\grps\to \Top_\grps$. Concretely, the functor $A$ consists of $G$-continuous functors $A_G\col \I_G\to \Top_G$ for all $G\in\grps$ such that for each homomorphism $\alpha\col G\to H$, there is a natural isomorphism of $G$-continuous functors
\begin{equation}\label{grpscompatibility}
\xymatrix{(\alpha^*)_\bullet \I_H\ar[r]^{A_H}\ar[d]_{\alpha^*} & (\alpha^*)_\bullet \Top_H \ar[d]_{\alpha^*}\\
\I_G\ar[r]^{A_G} & \Top_G}
\end{equation}
Denote this natural transformation by $\phi^A_\alpha\col \alpha^*\circ A_H\Rightarrow A_G\circ \alpha$.
These natural isomorphisms must be compatible with composition of group homomorphisms, so that $\phi^A_\alpha\circ\phi^A_\beta=\phi^A_{\alpha\beta}$. They must also satisfy the unit condition $\phi_{\id}=\id$

Morphisms of $\I_\grps$--spaces are natural transformations of enriched indexed functors. As such, a morphism $f\col A\to B$ consists of natural $G$--transformations $f_G\col A_G\to B_G$ for each $G\in \grps$, and these natural transformations are required to commute with the natural isomorphisms in (\ref{grpscompatibility}). In other words, we have a commutative diagram of natural transformations
\begin{equation}\label{phicompatiblenattrans}
\xymatrix{ \alpha^*\circ A_H\ar@{=>}[r]^{\alpha^*f_H} \ar@{=>}[d]_{\phi^{A}_\alpha} & \alpha^*\circ B_H\ar@{=>}[d]^{\phi^B_\alpha} \\
A_G\circ \alpha^*\ar@{=>}[r]^{f_G\circ\alpha} & B_G \circ \alpha^*}
\end{equation}
 We wish to emphasize that, as $f$ is a natural $G$-transformation, for each $V\in \I_G$, $f_G\col A_G(V)\to B_G(V)$ must be a $G$-map.  These morphisms make $\Igrpspa$ into a topologically enriched category.  
\end{defn}

\begin{rmk} \label{IgrpspaceIspacermk} Applying the compatibility condition of Equation (\ref{grpscompatibility}) to the homomorphism $\iota\col G\to 
e$ yields a natural isomorphism
\[ A_G\circ \iota^* \cong \iota^* A_e\]
which shows that if $V$ is a trivial $G$-representation (so that $V$ is contained in $\iota^*(\I_e)\subset \I_G$), then $A_G(V)$ has a trivial $G$-action.  This observation is the essential point to proving an equivalence between $\I$-spaces and $\I_\grps$-spaces, which we carry out in Appendix \ref{sec:globalize-this}. 
\end{rmk}

\begin{ex}\label{ex:globalsphere} The most fundamental example is the one-point compactification sphere functor $S$, which sends a representation $V$ to its one-point compactification $S^V$.  If $V\in \I_G$ is a $G$-representation and $\alpha\col H\to G$ is a homomorphism, we have a homeomorphism of $H$-spaces $\alpha^*(S^V)\cong S^{\alpha^*V}$, which is the heart of the required isomorphism (\ref{grpscompatibility}) of functors.  We will just call this the sphere functor.
\end{ex}

\begin{ex}\label{ex:globalcxcobord} The complex cobordism functors of Example \ref{ex:Gcxcobordism} also form an $\I_\grps$-space.  If $V\in \I_G$ is a $G$-representation and $\alpha\col H\to G$ is a homomorphism, the required homeomorphism of $H$-spaces $\alpha^*MU(V)\to MU(\alpha^*V)$ is the identity.
\end{ex}

We now turn to understanding the monoidal structures involved.  Recall from Section \ref{sec:indexedcattheory} that a monoidal $\Top_\grps$-fibration is a $\Top_\grps$-fibration $\mathcal{C}$ with a monoidal structure on each category $\mathcal{C}_G$ and strong monoidal transition functors $\alpha^*$.  Both $\Top_\grps$ itself and $\I_\grps$ are examples. The smash product of based spaces provides the monoidal structure on each category $\Top_G$ and it is clear that changing the group action via restriction is strongly monoidal.  The monoidal structure on $\I_G$ is given by direct sum together with the canonical identification $\RR^n\oplus\RR^m\cong \RR^{n+m}$.  More explicitly, $(\RR^n,\rho)\oplus (\RR^m,\mu)=(\RR^{n+m},\rho\oplus \mu)$, where $\rho\oplus\mu\col G\to O(n+m)$ is the homomorphism that sends $g\in G$ to the block matrix 
\[\begin{bmatrix} \rho(g) &0\\0& \mu(g) \end{bmatrix}.\]
Precomposition with a homomorphism $\alpha\col H\to G$ is strong monoidal: $\alpha^*(\rho\oplus \mu)=\alpha^*\rho\oplus\alpha^*\mu$.

The monoidal structure on $\I_\grps$ allows us to make the following definition.
\begin{defn}\label{globalspecdef} An \emph{orthogonal $\grps$-spectrum} is an $\I_\grps$-space $A$ together with an associative and unital action of the sphere functor.  In other words, there is a natural transformation of $\Top_\grps$ functors $\I_\grps\otimes_f \I_\grps\to \Top_\grps$ whose component at $(V,W)\in \I_\grps\otimes_f \I_\grps$ is 
\[\sigma_{(V,W)}\col A(V)\sma S^W\to A(V\oplus W) \]
\end{defn}

This definition uses an external notion of smash product on $\I_\grps$-spaces.  There is an internal smash product given by left Kan extension in the usual way.
\begin{defn} The (internal) \emph{smash product} of two $\I_\grps$-spaces $A$ and $B$ is the left Kan extension of their external smash product along the $\oplus$ functor $\I_\grps\otimes_f \I_\grps\to \I_\grps$.
\end{defn}

Such left Kan extensions exist because the categories $\I_\grps$ and $\Top_\grps$ are both well-behaved colimit-wise \cite{Shulman2011}.

Having  monoidal structures on $\I_\grps$ and $\Top_\grps$ allows us to define ring orthogonal $\grps$-spectra. If $A\col \I_\grps\to \Top_\grps$ is a lax monoidal functor of $\Top_\grps$-fibrations, then the unit $\eta\col S\to A$ and multiplication $\mu\col A(-)\sma A(-)\to A(-\oplus-)$ give $A$ the structure of an orthogonal $\grps$-spectrum via
\begin{equation}\label{laxmonoidalyieldsspectrum} A(V)\sma S^W\xrightarrow{\id\sma \eta} A(V)\sma A(W) \xrightarrow{\mu} A(V\oplus W). 
\end{equation}
Thus any lax monoidal functor $\I_\grps\to \Top_\grps$ gives a global orthogonal spectrum.  Both the sphere spectrum and complex cobordism, of Examples~\ref{ex:globalsphere} and \ref{ex:globalcxcobord}, are examples.  We will define two more such functors in the next section.

\section{$K$-theory and $\Spin^c$-cobordism}\label{sec:ktheoryspincobord}

For an arbitrary compact Lie group $G$, Michael Joachim constructs $G$-equivariant $K$-theory as an orthogonal $G$--spectrum \cite{Joachim2004}.  He also constructs an orthogonal version of $\MSpin^c_G$ and a map of orthogonal spectra realizing the complex orientation on $K$-theory.  We show that these orthogonal spectra for different groups $G$ in fact form an orthogonal $\grps$--spectrum.

For each compact Lie group $G$, Joachim constructs a $\I_\grps$--space $\KK_G$  using $\ZZ/2\ZZ$--graded  $C^*$--algebras.  He then shows that $\KK_G$ in fact represents $G$-equivariant complex $K$--theory.  Surprisingly, this is the only known $E_\infty$--version of equivariant complex $K$--theory when $G$ is a compact Lie group; for finite $G$, the $E_\infty$--structure on $K_G$ follows from \cite[Remark VII.4.4]{EKMM1997}.

\begin{defn}[Joachim {\cite{Joachim2004}}]\label{def:GKthry} Let $G$ be a compact Lie group.  Then $\KK_G\col \I_G\to \Top_G$ is a lax symmetric monoidal functor defined by $\KK_G(V)=\Hom_{C^*}(\csphere,\CC l_V\otimes \calK_V)$, where $\CC l_V$ is the Clifford algebra of $V$ and $\calK_V$ is the $G$-$C^*$--algebra of compact operators on $L^2(V)$.  The $C^*$--algebra $\csphere$ is the graded $G$-$C^*$--algebra of continuous functions on $\RR$ vanishing at infinity with trivial $G$--action.  The multiplication of lax symmetric monoidal structure arises from a comultiplication $\Delta\col \csphere\to \csphere\otimes\csphere$; the unit maps will be discussed in detail below.
\end{defn}

The $C^*$--algebra $\csphere$ is a $\ZZ/2\ZZ$ graded $C^*$--algebra with grading given by the even and odd functions. It admits a unique nontrivial $*$-homomorphism to $\CC$ given by evaluation at $0\in \RR$. Moreover, $\csphere$ is generated by the two functions $u(t)=(1+t)^\inv$ and $v(t)=t(1+t)^\inv$. Under the identification of $\csphere\otimes\csphere$ with the $C^*$--algebra of functions on $\RR^2$ vanishing at infinity, the comultiplication on $\csphere$ is defined on the generating functions by
\begin{equation*}
\Delta(u)(x,y)=\frac{1}{1+x^2+y^2} \textrm{\quad and\quad} \Delta(v)(x,y)=\frac{x+y}{1+x^2+y^2}.
\end{equation*}
This comultiplication is both coassociative and cocommutative, and this structure makes Joachim's construction of $\KK_G$ satisfy the associativity and commutativity required for a lax symmetric monoidal functor.

We restrict to the category of finite groups so that we may continue to use our skeletal category $\I_\grps$.

\begin{thrm}\label{ktheoryisglobal} The $G$--functors $\KK_G$ of Definition \ref{def:GKthry} define an orthogonal $\grps$--spectrum.
\end{thrm}

The heart of this statement is in understanding the interplay of $\KK_G$ with the restriction functors $\alpha^*$. 

\begin{proof}
Given $\alpha\colon H\to G$ and an $G$--representation $V\in \I_G$, we show that \[\KK_H(\alpha^*V)=\alpha^*\KK_G(V).\] for any $\alpha\col H\to G$. This follows from examination of how the equivariance arises in Joachim's construction. 

More explicitly, for $V=(\RR^n,\rho)$, the action of $G$ on the space 
\[\KK_G(V)=\Hom_{C^*}(\csphere,\CC l_V\otimes \mathcal{K}_V)\]
 is exclusively via the action of $G$ on $\CC l_V\otimes \mathcal{K}_V$,  since the $C^*$--algebra $\csphere$ is fixed. Hence we check that $\alpha^*(\CC l_V\otimes \mathcal{K}_V)$ is equal to $\CC l_{\alpha^*V}\otimes \mathcal{K}_{\alpha^*V}$ as $H$-$C^*$--algebras. As $\alpha^*(\CC l_V\otimes \mathcal{K}_V)= \alpha^*\CC l_V\otimes \alpha^* \mathcal{K}_V$, we check the equality for the Clifford algebra and compact operator pieces separately.   The action on the Clifford algebra $\CC l_V$ is simply the action of $G$ on $V$ extended to the tensor algebra; this action passes to the quotient $\CC l_V$ because $G$ acts via the orthogonal group on $\RR^n$ and thus preserves the standard inner product.  Hence the action of $H$ on $\alpha^*\CC l_V$ coincides with the extension of the action of $H$ on $\alpha^*V$ to the tensor algebra.  Thus $\CC l_{\alpha^*V}=\alpha^*{\CC l_V}$ as $H$-$C^*$--algebras.  

We next check that $\mathcal{K}_{\alpha^*V}=\alpha^*\mathcal{K}_V$.  The $G$-$C^*$--algebra $\mathcal{K}_V$ is the $C^*$--algebra of compact operators on the completion of the pre-Hilbert space of $L^2$ functions vanishing at infinity on $V$.  Again, since the $G$-action on $\RR^n$ preserves the inner product, the $C^*$-algebra inherits a $G$-action.

To be concrete,  $G$ acts on the space of functions from $V$ to $\CC$ by conjugation, which since $\CC$ is $G$ fixed, reduces to $g.f\mapsto f(g^{-1}(-))$.  Because $G$ preserves the inner product on $V$, it preserves norms in $V$ and thus the property of vanishing at infinity.  Similarly, the $G$--action preserves the $L^2$ norm on the function space: for an $L^2$ function $f\colon V\to \CC$ and an element $g\in G$, we find
\[\int_V \langle f(g^{-1}v),f(g^{-1}v)\rangle=\int_V\langle f(v),f(v)\rangle\]
since $g^{-1}$ acts invertibly on $V$.  Thus we have an $G$--action on $L^2(V)$, which induces an $G$--action by conjugation on operators from $L^2(G)$ to itself. This action preserves compact operators: an operator $T$ is compact if and only if the image of any bounded sequence under $T$ contains a convergent subsequence.  If $T$ is such an operator, then $g\circ T\circ g^{-1}$ is also compact.  For consider a bounded sequence $\{f_i\}$ of functions in $L^2(V)$.  Since $G$ preserves the $L^2$ norm, the sequence $\{ g^{-1}.f_i\}$ is also bounded, so its image $\{T(g^{-1}.f_i)\}$ contains a convergent subsequence, which is Cauchy.  Thus the corresponding subsequence of $\{gT(g^{-1}.f_i)\}$ is Cauchy and thus converges by completeness.

Hence we see that the $G$--action on $\mathcal{K}_V$ is induced by the fact that $G$ acts via orthogonal transformations on $\RR^n$ and thus preserves the inner product structure.  Tracing through this action shows that the $H$--inner product on $\alpha^*V$ is simply the restriction of the $G$--inner product to the image of $\alpha$ inside $G$.  It is thus apparent that $\alpha^*\mathcal{K}_V= \mathcal{K}_{\alpha^*V}$.  This shows that $\KK$ defines an $\I_\grps$-space 
\[\KK\col \I_\grps \to \Top_\grps.\]

It remains to show that $\KK$ is in fact an orthogonal $\grps$-spectrum.  In fact $\KK$ is a lax monoidal functor of $\Top_\grps$-enriched categories $\KK\col \I_\grps\to \Top_\grps$ in the sense of Definition \ref{monoidalfibrdefn}.  Joachim shows that for each fixed $G$, the functor $\KK_G\col \I_G\to \Top_G$ is lax monoidal; the discussion of monoidal structures in \cite{Shulman2011} implies that for our $\Top_\grps$ fibrations $\I_\grps$ and $\Top_\grps$, a lax monoidal functor of $\Top_\grps$ fibrations is equivalent to a lax monoidal functor $\I_G\to \Top_G$ for each $G$, plus appropriate compatibility for the transformations $\phi$.  Since the transformations $\phi^\KK$ are the identity on the point-set level, we certainly have all the required compatibility.  Hence the structure maps of Definition~\ref{globalspecdef} are given by composing the unit and multiplication natural transformations of the lax monoidal functor $\KK$, as in Equation (\ref{laxmonoidalyieldsspectrum}).
\end{proof}

We next turn to $\MSpin^c$.  Again, the construction of Joachim \cite{Joachim2004} generalizes to a global spectrum.   Since his construction is in still in the framework of $C^*$-algebras, we first establish some terminology.

For an inner product space $V\in\I_G$, the subgroup $Pin^c_V\subset \CC l_V$ is generated by the unit sphere $S(V)\subset \CC l_V$ together with the elements of the unit circle $S^1\subset \CC\subset \CC l_V$.  The $\ZZ/2$-grading of $\CC l_V$ restricts to a grading homomorphism $\xi\col Pin^c_V\to \ZZ/2$ and $Pin^C_V$ acts on $\CC l_V$ via conjugation twisted by $\xi$.  This action restricts to an isometry on  $V\subset \CC l_V$ and thus defines a group homomorphism $\varrho_V\col Pin^c_V\to O_V$ which is known to be surjective.

Let $B(\CC l_V\otimes L^2(V))$ be the space of bounded operators on $\CC l_V\otimes L^2(V)$ under the strong $*$-topology. Note that the space of compact operators on $\CC l_V\otimes L^2(V)$ is isomorphic to $\CC l_V\otimes \calK_V$. Let $\UU_V\subset B(\CC l_V\otimes L^2(V))$ be the subgroup of unitary operators and scalars, which comes with a grading homomorphism $\xi\col\UU_V\to \ZZ/2$.  The group $\UU_V$ acts on $B(\CC l_V\otimes L^2(V))$ via conjugation with twisting by $\xi$, and this action factors through the projective group $\PP\UU_V=\UU_V/S^1$.    We construct a map $j_V\col O_V\to \PP\UU_V$ by associating to each $e\in Pin^c_V$ the unitary operator $U_e\in B(\CC l_V\otimes L^2(V))$ defined by
\[ U_e(v\otimes f)=ev\otimes (f\circ \varrho_V(e)^\inv) \]
for $v\in \CC l_V$ and $f\in L^2(V)$.  This defines a homomorphism $Pin^c_V\to \UU_V$.  The kernel of the homomorphism $\varrho_V\col Pin^c_V\to O_V$ maps to $S^1\subset \UU_V$; the quotient thus gives the homomorphism $j_V\col O_V\to \PP\UU_V$.

Joachim's model of $\MSpin_G^c$ is the following.
\begin{defn}[{\cite[Definition 6.3]{Joachim2004}}]\label{mspindef} The equivariant spectrum $\MSpin_G^c$ is modeled by the lax monoidal functor $\MMSpin_G^c\col \I_G\to \Top_G$ 
which is defined as $\MMSpin_G^c(V)={\PP\UU_V}_+\sma_{O_V}S^V.$  The  monoidal structure is induced by the natural maps ${\UU_V}_+\sma{\UU_W}_+\to \UU_{V\oplus W}$ and $S^V\sma S^W\to S^{V\oplus W}$ and the unit is given by
\begin{equation}\label{spinunitmap}\eta\col S^V\cong {O_V}_+\sma_{O_V}\xrightarrow{j_V\sma\id} {\PP\UU_V}_+\sma_{O_V} S^V.
\end{equation}
\end{defn}

As $G$ varies, these models $\MMSpin^c_G$ fit together to form an orthogonal $\grps$-spectrum.
\begin{thrm}
The functors $\MMSpin_G^c$ of Definition \ref{mspindef} form a lax monoidal functor $I_\grps\to \Top_\grps$ and thus define an orthogonal $\grps$-spectrum.
\end{thrm}

\begin{proof}
As in the proof of Theorem \ref{ktheoryisglobal}, the change of groups functors $\alpha^*$ are the identity on the underlying spaces involved.  Again, this is because the inner product is assumed to be equivariant, and so the definitions of unitary operators and the like are all respected by the group actions in question.  Thus $\MSpin^c$ defines an $\I_\grps$-space, and the lax monoidal structure follows from this structure at each group individually.
\end{proof}

Finally, we note that there is a map of lax monoidal $\I_\grps$-spaces $\MMSpin^c\to \KK$ which, for $V\in\I_G$, is induced by the map 
\[\tilde{\gamma}_V\col {\UU_V}_+\sma S^V\to \Hom_{\CC^*}(\csphere,\CC l_V\otimes \calK_V)\]
 defined by $\tilde{\gamma}_V(U,v)=\xi(U)U\eta_V(v)U^*$.  The map $\eta$ of Equation (\ref{spinunitmap}) is $O_V$-equivariant by construction, so $\tilde{\gamma}_V$ induces the desired map
\[\gamma_V\col {\PP\UU_V}_+\sma_{O_V} S^V\to \Hom(\csphere,\CC l_V\otimes \calK_V)\]
Note that we are using the identification of $\CC l_V\otimes \calK_V$ with compact operators on $\CC l_V\otimes L^2(V)$ in order to define the action of $U\in\UU_V$ on $\CC l_V\otimes \calK_V$.  Again, since all the change-of-group functors $\alpha^*$ are the identity on underlying spaces, the maps $\gamma_V$ determine a natural transformation of functors of $\Top_\grps$-indexed categories.  As Joachim proves that at each group the natural transformation $\gamma$ is a model for the Atiyah--Bott--Shapiro orientation of complex K-theory \cite[Theorem 6.9]{Joachim2004}, we arrive at the following conclusion.

\begin{thrm}
The Atiyah--Bott--Shapiro orientation $\MSpin^c\to K$ extends to an orientation of orthogonal $\grps$-spectra.
\end{thrm}

\appendix

\section{$\I_G$--spaces and $\I$--$G$spaces.}\label{singlegroup}

These appendices are devoted to showing the equivalence between $\I$-spaces and $\I_\grps$-spaces mentioned in Remark \ref{IgrpspaceIspacermk}.  In this appendix, we warm up by proving that $\I_G$-spaces are equivalent to  $\I$-$G$spaces for a single group $G$.  This is an orthogonal version of a paper of Shimakawa's \cite{Shimakawa1991}, which proves that $\Gamma_G$-spaces and $\Gamma$-$G$spaces are equivalent.  A version of the equivalence we prove is shown in \cite{MM2002}, but without using the enriched category structure. For a fixed group $G$, this equivalence of categories is a key ingredient in comparing different models of genuine orthogonal $G$-spectra.  In Appendix \ref{sec:globalize-this}, we turn to the result for $\I_\grps$-spaces.

Let $G\Top$ be the category of based $G$-spaces and based equivariant maps between them; this category is enriched in spaces but not in $G$-spaces.  Denote by $\I$ the category $\I_e$ corresponding to the trivial group; this has objects the inner product spaces $\RR^n$.  In Section \ref{basicdefn}, we defined $\I_G$-spaces; we now define $\I$-$G$spaces, our other category of interest.

\begin{defn} An \emph{$\I$--$G$\,space} is a continuous functor $X\col \I\to G\Top$.  
%An \emph{$\I_G$--space} is a $G$--continuous functor $A\col \I_G\to \Top_G$. 
Morphisms of $\I$--$G$spaces are continuous natural transformations.
%; morphisms of $\I_G$--spaces are $G$--continuous natural transformations.  Hence the categories $\IGspa$ and $\IsubGspa$ are both topologically enriched.
\end{defn}
Note that the categories of $\I$-$G$spaces and $\I_G$-spaces are both topologically enriched.

There  is a fully faithful functor $\iota\col \I\to \I_G$ given by sending an object $\RR^n\in \I$ to the object $(\RR^n,\iota)\in \I_G$, where (by abuse of notation) $\iota$ denotes the unique homomorphism $G\to O(n)$ that factors through the trivial group.  In Appendix \ref{sec:globalize-this}, we will have call to think of this functor as precomposition with the unique homomorphism $\iota\col G\to e$.
The functor $\iota$ induces a restriction
\[R\col \IsubGspa \to \IGspa,\]
which sends an $\I_G$--space $A$ to its restriction to the trivial representations. 
\begin{rmk}\label{rmk:trivialrepsaretrivialyo}
Any map $f\col \RR^m\to \RR^n$ between trivial representations is $G$--fixed, and thus $f$ must be sent to a $G$--fixed map $A(f)\col A(\RR^m)\to A(\RR^n)$.  This implies that the restriction of $A$ to the category $\I$ in fact lands in $G\Top\subset\Top_G$.
\end{rmk}

In the remainder of this section, we prove the following theorem.
\begin{thrm}\label{singlegroupequivalence} The functor $R\col \IsubGspa \to \IGspa$ is an equivalence of topological categories.
\end{thrm}
This is essentially the content of \cite[Lemma V.1.1]{MM2002}; our proof is a $G$-enriched version of \cite{MMSS}.

As is usual, the functor $R$ has a left adjoint `extension' functor $E\col \IGspa\to \IsubGspa$ given by left Kan extension.
\begin{defn}\label{defn:extensionG} For $X\in \IGspa$, define $EX\in \IsubGspa$ as the topological left Kan extension $EX=Lan_\iota X$ in the following diagram:
\[
\xymatrix{\I \ar[r]^{\iota}\ar[d]_X& \I_G\ar[dl]^{EX=Lan_\iota X}\\
\Top_G}
\]
\end{defn}
\begin{rmk} As in Remark \ref{rmk:trivialrepsaretrivialyo}, we point out that any functor from $\I\to \Top_G$ must land in the subcategory $G\Top$, which legitimates our regarding $X$ as a functor $\I\to \Top_G$ instead of $\I\to G\Top$.
\end{rmk}

For ease of notation, let $V=(\RR^n,\rho)$ be an object of $\I_G$. Explicitly, the value of $EX$ at $V$ is given by the tensor product of functors
\[
EX(V)=\I_G(-,V)\otimes_{\I}X(-)=\coprod_{n}\I_G(\RR^n,V)\times X(\RR^n)/\sim
\]
where the equivalence relation is given by setting $[st,x]\sim [s,t_*x]$ for a map $t\col \RR^m\to \RR^n$ in $\I$.   $G$ acts diagonally on $\I_G(\RR^n,V)\times X(\RR^n)$; this action is compatible with the equivalence relation because $t$, as a map between trivial representations, induces an equivariant map $t_*$.  Furthermore, it is clear that this construction is functorial in $X$, either from the universal property of left Kan extension or via a direct check using the explicit definition.  

It is straightforward to check that $EX$ is a $G$--continuous functor. 
Let $f\col V\to W$ be a map in $\I_G$, let $g$ be an element of $G$ and $[s,a]\in EX(V)$.  By the definition of the $G$--action on morphisms in $\Top_G$,
\begin{align*}
(g.EX(f))[s,a] &= g(EX(f)[g^\inv sg,g^\inv a]\\
 &=gEX(f)[g^\inv s,g^\inv a]
\end{align*}
since $G$ acts trivially on the source of $s$. Further unraveling of definitions shows 
\begin{align*}
gEX(f)[g^\inv s, a] &= [gfg^\inv sg^\inv,gg^\inv a]\\
 &=[gfg^\inv s,a]
\end{align*}
where the second equality again follows from the trivial $G$--action on the source of $s$; hence $gEX(f)g^\inv=EX(gfg^\inv)$.

Since $E$ is defined by left Kan extension, the functors $R$ and $E$ are adjoint.  Furthermore, since $\iota\col \I \to \I_G$ is the inclusion of a full subcategory, the unit of this adjunction $\Id_{\IGspa}\to RE$ is a natural isomorphism by \cite[Proposition 3.2]{MMSS}.  Thus the following lemma completes the proof of the theorem.
\begin{lemma}
The counit $\epsilon\col ER\to \Id_{\IsubGspa}$ is a natural isomorphism.
\end{lemma}

\begin{proof}
Let $A\col \I_G\to \Top_G$ be an object of $\IsubGspa$.  For a $G$--representation $V=(\RR^d,\rho)$, the space $ERA(V)$ is given by 
\[ERA(V)=\coprod_n \I_G(\RR^n,V)\times A(\RR^n)/\sim\]
and the $V$th-component of ${\epsilon_{A}}$ is given by $[s,a]\mapsto s_*a$, where $s_*\col A(\RR^n)\to A(V)$ is the image of $s\in\I_G(\RR^n,V)$ under $A$. 
Note that functoriality of $A$ shows that $\epsilon_A$ is well-defined.

For each $A\in \IsubGspa$, $\epsilon_A$ must be a $G$--natural transformation; thus each component of ${\epsilon_A}$ must be a $G$-map.  This follows from the fact that $A$ preserves the $G$-enrichment, so that $A(g.s)=g.s_*$.

We now show $\epsilon_A$ is an isomorphism.
Suppose $V=(\RR^d,\rho)$.  Let $\RR^d\xarr{i} V\xarr{i^\inv} \RR^d$ be the (non-equivariant) maps in $\I_G$ that are the identity on the underlying vector space $\RR^d$. We prove that the counit ${\epsilon_A}\col ERA(V)\to A(V)$ is an isomorphism by defining a continuous inverse $\nu$.  For $a\in A(V)$, let 
\[\nu_A(a)=[i,i^\inv_*a].\]
Continuity of $i^\inv_*$ implies $\nu$ is continuous. It follows directly from the definition of $\epsilon$ and $\nu$ that $\epsilon_A\circ \nu_A\col A(V)\to A(V)$ is the identity map.  The other composition is also the identity: given $[s,a]\in ERA(V)$,
\begin{align*}
\nu_A(\epsilon_A[s,a]) &= [i,i^\inv_* s_*a]\\
 &\sim [i(i^\inv s),a]\\
 &=[s,a].
\end{align*}
Thus $\epsilon$ is a natural isomorphism. Moreover, since $\nu$ provides a point-set inverse to the equivariant map $\epsilon$, $\nu$ is also equivariant.
\end{proof}
We conclude that $R$ and $E$ provide an adjoint equivalence of the categories $\IGspa$ and $\IsubGspa$.

\section{The global version}\label{sec:globalize-this}

We now turn to the global version, using our definitions from Section \ref{sec:defnglobalspec}. 
The category of $\Igrpspa$ has been defined in Definition \ref{defn:Igrpspa}; we now define the other category in our comparison.
\begin{defn} An \emph{$\I$--space} is a continuous functor $X\col \I\to \Top$.  
Morphisms of $\I$--spaces are continuous natural transformations. 
\end{defn}

The global analogue of Theorem \ref{singlegroupequivalence} is the following.

\begin{thrm}\label{thrm:globalequiv} The category of $\I_\grps$--spaces is equivalent to the category of $\I$--spaces.  This equivalence is induced by an adjoint pair of functors $R\col \Igrpspa\to \Ispa$ and $E\col \Ispa\to \Igrpspa$.
\end{thrm}

Recall that an $\I_\grps$--space $A$ can be thought of as a collection of suitably compatible functors 
\[A_G\col \I_G\to \Top_G.\]
 In particular, we have a functor $A_e\col \I\to \Top$ corresponding to the trivial group $e$.  The map $A\mapsto A_e$ gives a well-defined forgetful functor $R\col \Igrpspa\to \Ispa$; this is one of the functors in our equivalence.

The functor $E\col \Ispa\to \Igrpspa$ ultimately comes from the left Kan extensions of Definition \ref{defn:extensionG}. Let $X\col \I\to \Top$ be an $\I$--space.  For any given finite group $G$, Definition \ref{defn:extensionG} provides a functor $E_GX\col \I_G\to \Top_G$ via left Kan extension.  We first must show that these functors fit together to give a well-defined functor of enriched indexed categories $EX\col \I_\grps\to\Top_\grps$; the required enriched indexed functor is then defined by setting $(EX)_G=E_GX$.
\begin{prop}Given an $\I$--space $X$, the functors $E_GX$ of Definition \ref{defn:extensionG} are the components of an $\I_\grps$--space.
\end{prop}
\begin{proof}
We must define the natural isomorphism  $G$--functors 
\[\phi^{EX}_\alpha\col \alpha^* \circ E_HX\to E_GX\circ \alpha^*\]
corresponding to the homomorphism $\alpha$. Let $V\in \I_H$.  Because $X(\RR^n)$ has trivial group action, we find that $\alpha^*(E_H X(V))=E_GX(\alpha^*V)$ as $G$--spaces.
It is clear that this identification is natural in $V$ and respects composition of group homomorphisms.  Hence the collection $\{E_GX\}$ forms an $\I_\grps$--space.
\end{proof}

We now turn to the proof of Theorem \ref{thrm:globalequiv}.  We begin by defining the unit and counit necessary for the adjunction. The unit has essentially the same structure as the unit in of the adjunction from Section \ref{singlegroup}.
\begin{defn} Let $\eta\col \Id_{\Ispa}\to RE$ be the natural transformation defined using the natural isomorphism of $\I$-spaces $X\xrightarrow{\cong} REX=E_eX$ that comes from the fact that $\I_e=\I$.  This is evidently natural in $X$, and provides a natural isomorphism $\eta\col \Id_{\Ispa}\to RE$.
\end{defn}
\begin{defn}\label{defn:globalcounit}
Let $\epsilon\col ER\to \Id_{\Igrpspa}$ be the natural transformation defined as follows.  Let $\iota\col G\to e$ be the unique homomorphism.  For $A\in\Igrpspa$, let $\phi^A_\iota\col \iota^*A_e(\RR^n)\to A_G\circ\iota^*$ be the natural isomorphism of Diagram (\ref{grpscompatibility}) corresponding to $\iota$.. The indexed natural transformation $\epsilon_A\col ERA\to A$ has $V$th component $\epsilon_A\col ERA(V)\to A(V)$ defined by
\[ \epsilon_A[s,a]=s_*\phi^A_\iota(a)\]
for $(s,a)\in \I_G(\RR^n,V)\times A_e(\RR^n)$.
\end{defn}
\begin{lemma} Definition \ref{defn:globalcounit} produces a well-defined continuous natural transformation \[\epsilon\col ER\to \Id_{\Igrpspa}.\]
\end{lemma}
\begin{proof}We must check well-definedness at several levels.
  First, we show that each component of ${\epsilon_A}$ is well defined on the equivalence classes in $ERA(V)$.  We must then check that $\epsilon_A$ is a morphism of enriched indexed functors.  Finally, we show that $\epsilon$ is natural with respect to morphisms of $\Igrpspa$.

The fact that ${\epsilon_A}$ is well defined on equivalence classes comes from the naturality of $\phi^A_\iota$.  Specifically, let $t\col \RR^m\to \RR^n$, $s\col \RR^n\to V$ and $a\in A_e(\RR^m)$ so that $[st,a]\sim [s,t_*a]$. Since $\phi^A_\iota$ is natural with respect to maps of trivial representations,
\begin{align*}
{\epsilon_A}(st,a)&=s_*t_*{\phi^A_\iota} a\\
&=s_*{\phi^A_\iota} (t_*a)\\
&={\epsilon_A}(s,t_*a)
\end{align*}
as required.

Next we show that $\epsilon_A$ is a morphism in the category $\Igrpspa$, that is,  a natural transformation of enriched indexed functors.  For each $G$, the argument of Appendix \ref{singlegroup} shows that our definition of $\epsilon_A$ restricts to a natural transformation of continuous $G$-functors $(ERA)_G=E_GRA\to A_G$.  We must show that these natural transformations commute with the compatibility transformations $\phi^A_\alpha$ and $\phi^{ERA}_\alpha$ associated to a homomorphism $\alpha\col G\to H$.  That is, for each $\alpha\col G\to H$, we must show that we have a commutative diagram of natural transformations
\[
\xymatrix{ \alpha^*\circ E_H RA\ar@{=>}[r]^{\alpha^*\circ \epsilon_A} \ar@{=>}[d]_{\phi^{ERA}_\alpha} & \alpha^*\circ A_H\ar@{=>}[d]^{\phi^A_\alpha} \\
E_G RA\circ \alpha^*\ar@{=>}[r]^{\epsilon_A\circ\alpha} & A_G \circ \alpha^*}
\]
as in Diagram (\ref{phicompatiblenattrans}).

Consider $V\in \I_H$.  The map 
\[{\phi^A_\alpha}\circ (\alpha^*\circ \epsilon_A) \col \alpha^* E_H RA(V)\to A_G(\alpha^* V)\]
sends an element $[s,a]$ to the image of $a\in \alpha^*A_e(\RR^n)$ under the map displayed along the top and right of the following diagram:
\[
\xymatrix{\alpha^*A_e(\RR^n)\ar[r]^{\phi_{\iota_H}} \ar[rd]_{\phi_{\iota_G}} & \alpha^* A_H(\RR^n) \ar[r]^{s_*}\ar[d]^{\phi_\alpha}& \alpha^* A_h(\RR^n)\ar[d]^{\phi_\alpha}\\
 & A_G(\RR^n)\ar[r]^{s_*} &A_G(\alpha^* V)
}
\]
The map \[ {\epsilon_A}\circ {\phi^{ERA}_\alpha} \col \alpha^*\circ E_HRA(V)\to A_G(\alpha^*V)\]
is given by sending $(s,a)\in \I_H(\RR^n,V)\times A_e(\RR^n)$ to the image of $a$ under the composite along the left and bottom of the same diagram.  Commutativity of this diagram follows from naturality of $\phi_\alpha$ and the compatibility requirement on the $\phi$'s: note that $\iota_G=\iota_H\circ \alpha$. \

Lastly, we must show that $\epsilon$ is natural with respect to morphisms in $\Igrpspa$.  Let $f\col A\to B$ be such a morphism.  Consider $V\in \I_G$. We show that the components of $(f\circ\epsilon_A)$ and $\epsilon_B\circ ERf$ at $V$ are equal.  Consider the following diagram:
\[
\xymatrix{ A_e(\RR^n)\ar[r]^{\phi^A_\iota}\ar[d]_{f_e} & A_G(\RR^n)\ar[r]^{s_*}\ar[d]_{f_G} & A_G(V)\ar[d]_{f_G}\\
 B_e(\RR^n)\ar[r]^{\phi^B_\iota} & B_G(\RR^n)\ar[r]^{s_*} & B_G(V)}
\]
The left square commutes because of the compatibility of $f$ with the natural isomorphisms $\phi^A$ and $\phi^B$ as in Diagram (\ref{phicompatiblenattrans}).  The right square commutes by naturality of $f_G$. If $[s,a]\in \I_G(\RR^n, V)\times A_e(\RR^n)$, then $f(\epsilon_A[s,a])$ is the image of $a$ under  the top and right maps in this diagram; $\epsilon_B(ERf[s,a])$ is the image of $a$ under the left and lower maps.  Thus  $\epsilon$ is natural in maps $f\col A\to B$ of $\I_\grps$--spaces. 
 This completes the proof that $\epsilon$ is a well-defined natural transformation $\epsilon\col ER\to \Id_{\Igrpspa}$.
\end{proof}

To complete the proof of Theorem \ref{thrm:globalequiv}, we must show that $\epsilon$ is a natural isomorphism.  We must also show that $\epsilon$ and $\eta$ satisfy the triangle identities.

\begin{proof}[{Proof of Theorem \ref{thrm:globalequiv}}]
We define an inverse to $\epsilon$.  For $V=(\RR^d,\rho)\in \I_G$, let $\RR^d\xrightarrow{i} V\xrightarrow{i^\inv} \RR^d$ be the non-equivariant isomorphisms given by forgetting the $G$--action on $V$.
 For each $A\in\Igrpspa$, define $\nu_A\col A\to ERA$ at $V\in \I_G$ by
\[\nu_A(a)=[i,\phi^\inv_\iota i^\inv_*a].\]
This map is continuous because $i^\inv_*$ and $\phi^\inv_\iota$ are both continuous.

Consider the composites $\epsilon_A\nu_A$ and $\nu_A\epsilon_A$: for $a\in A_G(V)$ and $[s,a]\in E_GA(V)$
\begin{align*}
\epsilon_A(\nu_A(a)) &=i_*\phi_\iota (\phi^\inv_\iota i^\inv_* a)\\
  &=a.
\end{align*}
Naturality of $\phi_\iota$ implies $ (i^\inv s)_*\phi_\iota=\phi_\iota(i^\inv s)_*$, which provides the third equivalence in the calculation
\begin{align*}
\nu_A(\epsilon_A[s,a]) &=[i,\phi^\inv_\iota i^\inv_*(s_*\phi_\iota a)]\\
&=[i,\phi^\inv_\iota(i^\inv s)_*\phi_\iota a]\\
&=[i,\phi^\inv_\iota\phi_\iota (i^\inv s)_* a]\\
&=[i,(i^\inv s)_* a]\\
&\sim[ii^\inv s, a]\\
&=[s,a].
\end{align*}
Thus both composites are the identity; we further need $\nu$ to be an enriched indexed functor.  It is standard (see \cite[page 16]{MacLane1998}) to show that component-wise inverses to a natural transformation define an inverse natural transformation; thus, at each $G$, $\nu_A\col A_G\to E_GRA$ is an inverse $G$--natural transformation to $\epsilon_A$. Compatibility of $\nu_A$ with the natural isomorphisms $\phi$ as in Diagram (\ref{phicompatiblenattrans}) follows by a similar argument, as does the naturality of $\nu$ with respect to morphisms $A\to B$ in $\Igrpspa$.  Thus we conclude that $\epsilon\col ER\to\Id_{\Igrpspa}$ is a natural isomorphism.

Finally we show that $\eta$ and $\epsilon$ satisfy the the triangle diagrams, so that they are indeed adjoint.  We must show that the triangle diagrams
\[
\xymatrix{ R\ar@{=}[dr]\ar[r]^{\eta\Id_R} & RER\ar[d]^{\Id_R\epsilon} \\
& R}
\textrm{\qquad and\qquad }
\xymatrix{E\ar@{=}[dr]\ar[r]^{\Id_E\eta} & ERE\ar[d]^{\epsilon\Id_E}\\
& E}\]
commute.
If $A\in \Igrpspa$, $\RR^n\in\I$, and $a\in RA(\RR^n)=A_e(\RR^n)$, then the first diagram sends
\[ a \mapsto [\id,a] \mapsto \id_*{\phi^A_{\id_e}}_n(a)=a\]
Consider the second triangle diagram. Let $X\in\Ispa$, $V\in \I_G$ and $x\in X(\RR^n)$.  By the definitions of $\epsilon$ and $\eta$, the  class $[s,x]\in E(V)$ is sent to 
\[ [s,x]\mapsto [s,[\id,x]]
%\in \I_G(\RR^n,V)\times REX(\RR^n)
\mapsto s_*{\phi^{EX}_\iota}_n[\id,x]=[s,x]\]
Hence the triangle identities are satisfied, so $E$ and $R$ define and adjoint equivalence of categories
\[ \xymatrix{R\col \Igrpspa \ar@<1ex>[r]& \Ispa :\!E \ar@<1ex>[l]}.\]
\end{proof}

\begin{rmk}
To maintain consistency with the main body of this paper, we have required the maps in $\I_G$ and $\I_\grps$ to be linear isometric isomorphisms throughout, as is usual in defining equivariant orthogonal spectra.  However, all proofs in Appendices \ref{singlegroup} and \ref{sec:globalize-this} work just as well if we allow our maps to merely be linear isometries.
\end{rmk}

\bibliography{generalrefs}{}
\bibliographystyle{amsplain}

\end{document}